\documentclass[11pt,letterpaper]{article}
\usepackage[margin=1in]{geometry}

\usepackage{amsmath,amssymb,amsthm,mathrsfs,bbm,comment,units,enumitem,xcolor}
\usepackage[colorlinks, linktocpage]{hyperref}
\usepackage{pgf,tikz}
\usepackage{tikz-3dplot}
\usetikzlibrary{positioning}
\usetikzlibrary{decorations.pathreplacing}
\usetikzlibrary{arrows,shapes}
\usetikzlibrary{arrows.meta,bending,automata}

\renewcommand{\leq}{\leqslant}

\newcommand{\concat}{\overset{\frown}{}}
\newcommand{\rest}{\upharpoonright}

\newcommand{\tuple}[1]{\left\langle#1\right\rangle}

\newcommand{\win}{\uparrow}
\newcommand{\markwin}{\uparrow_{\mathrm{mark}}}

\newcommand{\prewin}{\uparrow_{\mathrm{pre}}}
\newcommand{\notprewin}{\not\uparrow_{\mathrm{pre}}}

\newcommand{\mc}{\mathcal}
\newcommand{\ran}{\operatorname{ran}}

\newcommand{\term}{\textbf}

\theoremstyle{definition}
\newtheorem{theorem}{Theorem}[section]
\newtheorem{definition}[theorem]{Definition}
\newtheorem{lemma}[theorem]{Lemma}
\newtheorem{corollary}[theorem]{Corollary}
\newtheorem{proposition}[theorem]{Proposition}
\newtheorem{conjecture}[theorem]{Conjecture}
\newtheorem{example}[theorem]{Example}

\newtheorem{question}[theorem]{Question}

\title{On strategies for selection games related to
countable dimension}
\author{Christopher Caruvana and Steven Clontz}

\begin{document}

\maketitle

\begin{abstract}
    Two selection games from the literature,
    \(G_c(\mc O,\mc O)\) and \(G_1(\mc O_{zd},\mc O)\),
    are known to characterize countable dimension among certain
    spaces. This paper studies their perfect- and limited-information
    strategies, and investigates issues related
    to non-equivalent characterizations of zero-dimensionality
    for spaces that are not both separable and metrizable.
    To relate results on
    zero-dimensional and finite-dimensional spaces,
    a generalization of Telg\'{a}rsky's
    proof that the point-open and finite-open games
    are equivalent is demonstrated.
\end{abstract}

\section{Introduction}

In the field of topological dimension theory, there are three standard notions of dimension:
the small inductive dimension, the large inductive dimension, and the covering dimension.
Though some spaces, like the disjoint union of \([0,1]^n\) over positive integers \(n\), are ``weakly'' infinite-dimensional,
the Hilbert cube \([0,1]^\omega\) is ``strongly'' infinite-dimensional
\cite[1.8.20]{engelking1978dimension}.
Hurewicz, in \cite{zbMATH02577822},
introduced the notion of countable-dimensional spaces as a first step to characterizing
infinite-dimensional spaces.

\begin{definition}
    A separable metrizable space is said to have
    \term{countable dimension} if it is the countable union of
    zero-dimensional subspaces.
\end{definition}

Since a countable-dimensional space can be written as a countable union of zero-dimensional subspaces,
we will be focusing our attention on these subspaces.
We use the following terminology to disambiguate between two common 
characterizations of zero dimension from the literature.

\begin{definition}
    A space is said to be \term{zero-ind} if it has a basis of clopen sets.
\end{definition}

\begin{definition}
    A space is said to be \term{zero-dim} or \term{0dim} if every open cover admits a refining open cover consisting of pairwise disjoint open sets.
\end{definition}

\begin{proposition}
    Every \(T_1\) \term{zero-dim} space \(X\) is \term{zero-ind}.
\end{proposition}
\begin{proof}
    Let \(U\) be an open neighborhood of \(x\) and consider the open
    cover \(\{U,X\setminus\{x\}\}\). A pairwise disjoint open refinement
    covering \(X\) includes a clopen subset of \(U\) containing \(x\).
\end{proof}

The ``ind'' above abbreviates ``inductive'' as this is the usual definition
of zero inductive dimension; we use just ``dim'' for covering dimension
to follow historical precedent. There is also a large inductive
(``Ind'') dimension defined in terms of closed subsets rather than
points;
in the context of normal spaces, covering dimension zero and this
large inductive dimension zero coincide
\cite[1.6.11]{engelking1978dimension};
we will not consider it further.
In the further restricted context of separable metrizable spaces, all three 
notions of dimension coincide \cite[1.7.7]{engelking1978dimension}.
However, in the general context of metrizable spaces, the small inductive 
dimension can differ from the covering dimension,
the first example of such a space coming from \cite{RoyPrabir1962}.
The natural question then is how different can the two notions be?
This general question of dimension spread has enjoyed steady progress
\cite{Kulesza1990,Kulesza2005,Mrowka1997,Mrowka2000,OSTASZEWSKI199095}.

We will be interested in the following generalization of
\(\sigma\)-compactness.

\begin{definition}
    Let \(\mc A\) be a collection (resp. property) of subsets of a space \(X\).
    Then \(X\) is said to be \term{\(\sigma\)-\(\mc A\)} provided there
    exists a countable collection \(\{A_n:n<\omega\}\) of sets in \(\mc A\) 
    (resp. sets satisfying \(\mc A\)) where \(X=\bigcup_{n<\omega} A_n\).
\end{definition}

\begin{question}
When is \(\sigma\)-zero-ind equivalent to \(\sigma\)-zero-dim?
\end{question}

Any answer to the following question is also an answer to the former.

\begin{question}
When does zero-ind imply \(\sigma\)-zero-dim?
\end{question}

As noted earlier, zero-ind and zero-dim (and therefore \(\sigma\)-zero-ind
and \(\sigma\)-zero-dim) are equivalent for separable metrizable spaces.
This assumption may be weakened by considering the following property.

\begin{definition}
A space is said to be
\term{strongly paracompact} if every open cover of
the space has a star-finite open refinement that covers the space,
that is, an open refinement covering the space such that each member
of the refinement meets only finitely-many other members.
\end{definition}

From the definition, one observes that every strongly paracompact
space is paracompact (every open cover has a locally-finite refinement
that's also a cover). It may also be shown
\cite[Corollary 5.3.11]{MR1039321} that every Lindel\"of \(T_3\) space,
such as a separable metrizable space, is strongly paracompact.

\begin{theorem}[{\cite[3.1.30]{engelking1978dimension}}]
Let \(X\) be strongly paracompact and \(T_2\).
Then \(X\) is zero-ind if and only if it is zero-dim.
\end{theorem}

Due to this, the discussion of ``zero dimension'' and
``countable dimension'' is unambiguous
in its usual context of separable metrizable spaces, but it seems more care
must be taken if strong paracompactness is not guaranteed, even if
the space is metrizable. As such, we prefer to refer to subspaces as
zero-dim or zero-ind as appropriate throughout this paper, and only use
zero/countable-dimensional when it is guaranteed these concepts are equivalent.

\section{Relative covering dimension}

It's important to note that zero-dim and zero-ind are both properties
of topological spaces, and thus a subset of a space must be considered
using its subspace topology. Consider then the following variation of
zero covering dimension for a subset which does not consider the subspace topology.
(According to \cite{megaritis2012relative}, for
paracompact spaces this definition coincides with a
definition of relative dimension given in
\cite{zbMATH01038637}.)

\begin{definition}
A subset \(Y\subseteq X\) is \term{relatively zero-dim to \(X\)},
\term{zero-dim\(_X\)}, or \term{0dim\(_X\)}
if for every cover of \(Y\) by sets open in \(X\),
there exists a pairwise disjoint refinement covering \(Y\) by sets open in \(X\).
\end{definition}

We proceed by first demonstrating that this is a stronger property
for a subset than zero-dim.

\begin{proposition}
For any \(Y\subseteq X\),
if \(Y\) is zero-dim\(_X\), then \(Y\) is zero-dim.
\end{proposition}
\begin{proof}
Assume \(Y\) is zero-dim\(_X\). Let \(\mc U\) be a cover of \(Y\) by
open subsets of \(Y\). Let \(\mc U'\) be a collection of open subsets of
\(X\) such that for each \(U\in \mc U\) there exists \(U'\in\mc U'\)
such that \(U=U'\cap Y\). Let \(\mc V'\) be a pairwise-disjoint open refinement
of \(\mc U'\) covering \(Y\). Let \(\mc V=\{V'\cap Y:V'\in\mc V'\}\).
It follows that \(\mc V\) is a pairwise disjoint refinement of \(\mc U\)
of sets open in \(Y\) that covers \(Y\), so \(Y\) is zero-dim.
\end{proof}

However, zero-dim\(_X\) is not equivalent, in general,
to the usual covering dimension of a subspace.

\begin{definition}
Let \(N=\mathbb R\times[0,\infty)\) have the
\term{tangent disc topology},
where points in \(\mathbb R\times(0,\infty)\) have their usual Euclidean
neighborhoods, and points in \(\tuple{x,0}\in\mathbb R\times\{0\}\) have neighborhoods
of the form 
\(U_{x,\epsilon}=\{\tuple{x,0}\}\cup B_\epsilon(\tuple{x,\epsilon})\)
for \(\epsilon>0\)
(where \(B_r(P)\) is the open ball of radius \(r\) centered at \(P\)).
\end{definition}

\begin{definition}
For \(B\subseteq\mathbb R\), let
\(N(B)=(B\times\{0\}) \cup (\mathbb R\times(0,\infty))\)
be the \term{bubble space} of \(B\), with the subspace topology
inherited from \(N=N(\mathbb R)\).
\end{definition}

\begin{example}
If \(B\subseteq\mathbb R\) is uncountable, then \(N(B)\) is
a \(T_{3.5}\) space with a subset \(B\times\{0\}\) which
is zero-dim but not \(\sigma\)-zero-dim\(_X\).

\end{example}

\begin{proof}
We first note that \(B\times\{0\}\) is discrete,
and any discrete space is zero-dim.

We proceed by showing that if \(B\times\{0\}\) is \(\sigma\)-zero-dim\(_X\),
then \(B\) is countable. Consider the open cover
\(\{U_{z,1}:z\in\mathbb R\}\) of any subset of \(B\).
For any pairwise disjoint open refinement of this cover,
only countably-many points \(\tuple{z,0}\) from the subset can be covered, as
any open subset of \(U_{z,1}\) containing \(\tuple{z,0}\)
must contain a distinct point in the countable set \(\mathbb Q^2\).
\end{proof}

Note that if \(|B|=\aleph_1\) and \(MA+\neg CH\) holds, then \(N(B)\) is
also \(T_4\) \cite{fleissner1980,bing1951metrization}, so the properties
zero-dim\(_X\) and zero-dim are consistently distinct for normal spaces.
Nonetheless, these notions do in fact coincide when considering metrizable
spaces.

\begin{theorem}
Let \(X\) be metrizable. Then for any \(Y\subseteq X\),
\(Y\) is zero-dim\(_X\) if and only if \(Y\) is zero-dim.
\end{theorem}

\begin{proof}
Assume \(Y\) is zero-dim. Let \(\mc U=\{U_\alpha:\alpha<\kappa\}\)
be a cover of \(Y\) by open
subsets of \(X\). Then \(\mc U'=\{U'_\alpha:\alpha<\kappa\}\)
where \(U'_\alpha=U_\alpha\cap Y\) for each \(\alpha<\kappa\) is a cover of \(Y\)
by open subsets of \(Y\). Let
\(\mc V'=\{V'_{\alpha,\beta}:\alpha<\kappa,\beta<\lambda_\alpha\}\)
where \(V'_{\alpha,\beta}\subseteq U'_\alpha\)
be a pairwise-disjoint refinement of \(\mc U'\)
by sets open in \(Y\) covering \(Y\).

By \cite[II.XI.21.2]{kuratowski2014topology} and the metrizability of \(X\),
there exists a pairwise disjoint collection
\(\mc V=\{V_{\alpha,\beta}:\alpha<\kappa,\beta<\lambda_\alpha\}\) of sets open
in \(X\) such that \(V'_{\alpha,\beta}=V_{\alpha,\beta}\cap Y\).
Let \(W_{\alpha,\beta}=V_{\alpha,\beta}\cap U_\alpha\), so
\(\mc W=\{W_{\alpha,\beta}:\alpha<\kappa,\beta<\lambda_\alpha\}\) is a
pairwise disjoint refinement of \(\mc U\) by sets open in \(X\).
Let \(y\in Y\), and choose \(\alpha<\kappa,\beta<\lambda_\alpha\) such that
\(y\in V'_{\alpha,\beta}\). Then both
\(y\in V'_{\alpha,\beta}\subseteq U_\alpha'\subseteq U_\alpha\)
and \(y\in V'_{\alpha,\beta}\subseteq V_{\alpha,\beta}\), so \(y\in W_{\alpha,\beta}\).
Thus \(\mc W\) covers \(Y\), and \(Y\) is zero-dim\(_X\).
\end{proof}

The use of \cite[II.XI.21.2]{kuratowski2014topology} is no coincidence,
as it was the technique utilized in \cite{babinkostova2021countable}
to obtain game-theoretic characterizations of countable dimension
among strongly paracompact metrizable spaces. Put another way, topological
games which deal with open covers of the entire space
more naturally characterize spaces in terms of
the relative dimension of their subsets, but this distinction is lost when
only studying metrizable spaces.
This is explored further in the following section.

\section{Perfect- and limited-information strategies}

Our goal is to study countable dimension in the context of the following two
games.

\begin{definition}
    Let \(\mc O\) collect the open covers of a space.
    The game
    \(G_c(\mc O,\mc O)\)
    is played by ONE and TWO. During each round \(n<\omega\),
    ONE chooses some \(\mc U_n\in\mc O\), and then
    TWO chooses some pairwise disjoint open refinement \(\mc V_n\) of \(\mc U_n\).
    TWO wins this game if \(\bigcup\{\mc V_n:n<\omega\}\) covers \(X\),
    and ONE wins otherwise.
\end{definition}

\begin{definition}
    Let \(\mc T\) be the topology of a space, and
    \(\mc O_{\mc A}=\{\mc U\subseteq\mc T:\forall A\in\mc A\exists U\in\mc U(A\subseteq U)\}\).
    Then the game
    \(G_1(\mc O_{\mc A},\mc O)\)
    is played by ONE and TWO. During each round \(n<\omega\),
    ONE chooses some \(\mc U_n\in\mc O_{\mc A}\), and then
    TWO chooses some open set \(V_n\in\mc U_n\).
    TWO wins this game if \(\{\mc V_n:n<\omega\}\) covers \(X\),
    and ONE wins otherwise.
\end{definition}

We will be particularly interested when \(\mc A\) collects the ``zero-dimensional''
subsets of a space, and write e.g. \(\mc O_{0\dim_X}\).

We also will consider a natural variation of this game.

\begin{definition}
    Let \(\mc T\) be the topology of a space,
    \(B_A=\{U\in\mc T:A\subseteq U\}\), and
    \(\mc N_{\mc A}=\{B_A:A\in\mc A\}\).
    Then the game
    \(G_1(\mc N_{\mc A},\neg\mc O)\)
    is played by ONE and TWO. During each round \(n<\omega\),
    ONE chooses some \(B_{A_n}\in\mc N_{\mc A}\), and then
    TWO chooses some open set \(V_n\in B_{A_n}\).
    TWO wins this game if \(\{\mc V_n:n<\omega\}\) fails to cover \(X\),
    and ONE wins otherwise.
\end{definition}

These are both examples of \term{selection games} \(G_1(\mc A,\mc B)\)
(ONE chooses \(A_n\in\mc A\), TWO chooses \(b_n\in A_n\), TWO wins if
\(\{b_n:n<\omega\}\in\mc B\))
used to characterize many topological properties
\cite{DIAS201558}.

\begin{definition}
    Let \(G\) be a game where players choose from the set \(M\).
    Then \(\tau:M^{<\omega}\to M\) defines a \term{(perfect-information)
    strategy} for either player, where \(\tau(\tuple{m_0,\dots,m_N})\in M\)
    is the move selected by the strategy in response to the opponent
    choosing \(m_i\in M\) during round \(0\leq i\leq N\).
    If player \(P\) has a winning strategy that defeats every play
    of the opponent for \(G\), then we write \(P\win G\).
    
    Likewise, \(\tau:M\times\omega\to M\) defines a \term{Markov strategy}
    that makes its choice \(\tau(m,N)\) based on only the most recent move \(m\in M\) 
    of the opponent and the current round number \(N<\omega\), and
    \(\tau:\omega\to M\) defines a \term{predetermined strategy}
    that ignores the moves of the opponent and makes a choice \(\tau(N)\) 
    based only on the value of the current
    round \(N<\omega\). Then \(P\markwin G\) (resp. \(P\prewin G\)) means the player
    \(P\) has a winning Markov (resp. predetermined) strategy that defeats
    every play of the opponent for \(G\).
\end{definition}

So for example, \(ONE\notprewin G_1(\mc O,\mc O)\) characterizes the Rothberger
covering property (commonly expressed as the \term{selection principle}
\(S_1(\mc O,\mc O)\)) \cite{MR1378387}. Likewise,
\(TWO\markwin G_1(\mc O,\mc O)\) characterizes the countability of a
\(T_1\) space (see e.g. \cite{CLONTZ2019106815});
we will soon demonstrate a generalization of this result.

We now formalize the relationship between
    \(G_1(\mc O_{\mc A},\mc O)\) and
    \(G_1(\mc N_{\mc A},\neg\mc O)\).

\begin{definition}
Let \(G,H\) be games with players ONE and TWO.
Then we say \(G\) and \(H\) are \term{equivalent} provided:
    \begin{itemize}
        \item \(ONE\win G\) if and only if \(ONE\win H\).
        \item \(TWO\win G\) if and only if \(TWO\win H\).
        \item \(ONE\prewin G\) if and only if \(ONE\prewin H\).
        \item \(TWO\markwin G\) if and only if \(TWO\markwin H\).
    \end{itemize}

And we say \(G\) and \(H\) are \term{dual} provided:
    \begin{itemize}
        \item \(ONE\win G\) if and only if \(TWO\win H\).
        \item \(TWO\win G\) if and only if \(ONE\win H\).
        \item \(ONE\prewin G\) if and only if \(TWO\markwin H\).
        \item \(TWO\markwin G\) if and only if \(ONE\prewin H\).
    \end{itemize}
\end{definition}

The following may be proven from the techniques of
\cite{clontz2020dual}.

\begin{proposition}\label{dualGames}
    \(G_1(\mc O_{\mc A},\mc O)\) and
    \(G_1(\mc N_{\mc A},\neg\mc O)\) are dual for any \(\mc A\).
\end{proposition}
\begin{proof}
Let \(\mathbf C(X)\) collect the choice functions
\(f:X\to\bigcup X\) such that \(f(x)\in x\) for all
\(x\in X\).
To see that \(\mc N_{\mc A}\) is a reflection
of \(\mc O_{\mc A}\) we may verify:
  \begin{itemize}
    \item \(\ran f\in\mc O_{\mc A}\) for all \(f\in\mathbf C(\mc N_{\mc A})\), and
    \item for each \(\mc U\in\mc O_{\mc A}\) there exists
    \(f_A\in\mathbf C(\mc N_{\mc A})\) such that \(\ran{f_{\mc U}}\subseteq \mc U\).
  \end{itemize}
Therefore the games are dual.
\end{proof}

Note when \(\mc A=X^1\) collects
the singletons of \(X\), then 
\(G_1(\mc O_{\mc A},\mc O)\) is the Rothberger game and
\(G_1(\mc N_{\mc A},\neg\mc O)\) is the well-known point-open game.
These were shown to be dual for perfect-information strategies 
by Galvin in \cite{MR0493925}.

The following result shows that
\(TWO\markwin G_c(\mathcal O,\mathcal O)\) provides a very
natural characterization of ``countable relative dimension''
for an arbitrary topological space.

\begin{theorem}
\(TWO\markwin G_c(\mathcal O,\mathcal O)\) if and only if
\(X\) is \(\sigma\)-zero-dim\(_X\).
\end{theorem}
\begin{proof}
Let \(\tau\) be a winning Markov strategy for TWO.
Let \[X_n=\bigcap_{\mathcal U\in\mathcal O}\bigcup\tau(\mathcal U,n)\]

Let
\(\mathcal U\) cover \(X\), then \(\tau(\mathcal U,n)\) is a pairwise disjoint refinement
of \(\mathcal U\) and covers \(X_n\). Therefore \(X_n\) is zero-dim\(_X\).

Then consider \(x\in X\). If \(x\not\in X_n\) for all \(n<\omega\), choose
\(\mathcal U_n\in\mathcal O\) with \(x\not\in\bigcup\tau(\mathcal U_n,n)\).
Then \(\mathcal U_n\) is a successful counterattack to the winning strategy \(\tau\),
a contradiction. Therefore \(x\in X_n\) for some \(n<\omega\), and
\(X=\bigcup_{n<\omega} X_n\). Thus \(X\) is \(\sigma\)-zero-dim\(_X\).

Now assume \(X=\bigcup_{n<\omega}X_n\) with \(X_n\) zero-dim\(_X\).
Let \(\tau(\mc U,n)\) be a pairwise disjoint open refinement of \(\mc U\)
covering \(X_n\).
It follows that \(\tau\) is a winning Markov strategy for TWO.
\end{proof}

Likewise, Markov strategies for TWO in
\(G_1(\mathcal O_{\mc A},\mathcal O)\) also may characterize spaces
of countable dimension (or spaces which are countable unions of
anything you'd like).

\begin{lemma}\label{sigmaMark}
Let \(X\) be \(T_1\). Then
\(TWO\markwin G_1(\mathcal O_{\mc A},\mathcal O)\) if and only if
\(X\) is \(\sigma\)-\(\mc A\).
\end{lemma}

\begin{proof}
Let \(\tau\) be a winning Markov strategy for TWO.
Let
\[ X_n = \bigcap_{\mathcal U\in\mathcal O_{\mc A}} \tau(\mathcal U,n)\]

Suppose \(X_n\not\subseteq A\) for all \(A\in \mc A\).
Pick \(x_A\in X_n\setminus A\) for each \(A\in\mathcal A\).
If \(X\in\mc A\) we're done, so assume \(X\not\in\mc A\). Then
\(\mathcal U=\{X\setminus\{x_A\}:A\in\mathcal A\}\in\mc O_{\mc A}\).
But then \(\tau(\mathcal U,n)=X\setminus\{x_A\}\) for some \(A\in\mc A\). Thus
\(X_n\not\subseteq\tau(\mathcal U,n)\), contradiction.
Thus we have \(X_n\subseteq A_n\in\mc A\) for \(n<\omega\).

Then consider \(x\in X\). If \(x\not\in X_n\) for all \(n<\omega\), choose
\(\mathcal U_n\in\mathcal O\) with \(x\not\in\tau(\mathcal U_n,n)\).
Then \(\mathcal U_n\) is a successful counterattack to the winning strategy \(\tau\),
a contradiction. Therefore \(x\in X_n\) for some \(n<\omega\), and
\(X=\bigcup_{n<\omega} X_n=\bigcup_{n<\omega} A_n\).
Thus \(X\) is \(\sigma\)-zero-dim\(_X\).

Now assume \(X=\bigcup_{n<\omega}X_n\) with \(X_n\in\mc A\).
Let \(\tau(\mc U,n)\) be a member of \(\mc U\) that contains \(X_n\).
It follows that \(\tau\) is a winning Markov strategy for TWO.
\end{proof}

\begin{corollary}
Let \(X\) be \(T_1\). Then the following are equivalent.
\begin{itemize}
    \item 
        \(X\) is \(\sigma\)-zero-dim\(_X\).
    \item
        \(TWO\markwin G_c(\mathcal O,\mathcal O)\)
    \item
        \(TWO\markwin G_1(\mathcal O_{0\dim_X},\mathcal O)\)
    \item
        \(ONE\prewin G_1(\mathcal N_{0\dim_X},\neg\mathcal O)\)
\end{itemize}
\end{corollary}

We now turn to perfect-information strategies in these games.

\begin{proposition}\label{gdeltasigma}
    If \(\mathcal A\) has the property that for each \(A\in\mathcal A\) there exists
    a \(G_\delta\) \(A^\star\) with \(A\subseteq A^\star\in\mathcal A\), then
    \(ONE\win G_1(\mc N_{\mc A},\neg\mc O)\) if and only if
    \(ONE\prewin G_1(\mc N_{\mc A},\neg\mc O)\)
    if and only if \(X\) is \(\sigma\)-\(\mathcal A\).
\end{proposition}
\begin{proof}
    Of course, if \(X = \bigcup \{A_n:n\in\omega\}\) where \(A_n \in \mathcal A\)
    for each \(n\in\omega\), then One has a winning predetermined strategy.
    So we need only show that ONE having a winning strategy witnesses
    that \(X\) is \(\sigma\)-\(\mathcal A\).
    
    Suppose \(\tau\) is winning for ONE;
    let \(\tau'\) yield corresponding members of \(\mc A\).
    Then \(\tau'(\tuple{})\in\mc A\), so choose \(U_{\tuple{n}}\) open with
    \(\tau'(\tuple{n})\subseteq X_{\tuple{}}
    =\bigcap\{U_{\tuple n}:n<\omega\}\in\mc A\).
    Now if \(U_{s\rest n}\) is defined for \(s\in\omega^{n+1}\) and
    \(0<n\leq|s|\), note
    \(\tau'(\tuple{U_{s\rest 1},\dots,U_{s}})\in\mc A\), so choose
    \(U_{s\concat \tuple{n}}\) open with
    \(\tau'(\tuple{U_{s\rest 1},\dots,U_{s}})\subseteq X_s =
    \bigcap\{U_{s\concat\tuple n}:n<\omega\}\in\mc A\).
    
    We now show that \(X=\bigcup_{s\in\omega^{<\omega}} X_s\). If not,
    pick \(x\not\in X_s\) for all \(s\in\omega^{<\omega}\).
    We then may define \(f\in\omega^\omega\) such that
    \(x\not\in U_{f\rest n+1}\) for all \(n<\omega\).
    Finally, note that the counterattack \(\tuple{U_{f\rest 1},U_{f\rest 2},\dots}\)
    defeats \(\tau\), a contradiction.
\end{proof}

From duality we may obtain the following existing result.

\begin{corollary}[{\cite[Theorem 6]{BabinkostovaScheepersNote}}]
    If \(\mathcal A\) has the property that for each \(A\in\mathcal A\) there exists
    a \(G_\delta\) \(A^\star\) with \(A\subseteq A^\star\in\mathcal A\), then
    \(TWO\win G_1(\mc O_{\mc A},\mc O)\)
    if and only if \(X\) is \(\sigma\)-\(\mathcal A\).
\end{corollary}

Put together, we see the following.

\begin{corollary}
    For \(T_1\) spaces,
    if \(\mathcal A\) has the property that for each \(A\in\mathcal A\) there exists
    a \(G_\delta\) \(A^\star\) with \(A\subseteq A^\star\in\mathcal A\), then
    the following are all equivalent.
    \begin{itemize}
        \item
            \(X\) is \(\sigma\)-\(\mathcal A\).
        \item 
            \(TWO\win G_1(\mc O_{\mc A},\mc O)\).
        \item 
            \(TWO\markwin G_1(\mc O_{\mc A},\mc O)\).
        \item 
            \(ONE\win G_1(\mc N_{\mc A},\neg\mc O)\).
        \item 
            \(ONE\prewin G_1(\mc N_{\mc A},\neg\mc O)\).
    \end{itemize}
\end{corollary}

\begin{corollary}\label{gdeltacor}
    If every zero-dim\(_X\) subset of \(X\) is contained in a
    \(G_\delta\) zero-dim\(_X\) subset, then
    the following are all equivalent.
    \begin{itemize}
        \item
            \(X\) is \(\sigma\)-zero-dim\(_X\).
        \item
            \(TWO\markwin G_c(\mathcal O,\mathcal O)\).
        \item 
            \(TWO\win G_1(\mc O_{0\dim_X},\mc O)\).
        \item 
            \(TWO\markwin G_1(\mc O_{0\dim_X},\mc O)\).
        \item 
            \(ONE\win G_1(\mc N_{0\dim_X},\neg\mc O)\).
        \item 
            \(ONE\prewin G_1(\mc N_{0\dim_X},\neg\mc O)\).
    \end{itemize}
\end{corollary}

It would seem natural for the following conjecture to hold, but its
validity is currently an open question.

\begin{conjecture}\label{conjEq}
\(TWO\win G_c(\mc O,\mc O)\) may be added to
Corollary \ref{gdeltacor}.
\end{conjecture}

Furthermore, this would align with existing results on
metrizable spaces, given the following lemma.

\begin{lemma}[{\cite[4.1.19]{engelking1978dimension}}]
Let \(X\) be metrizable. Then every zero-dim subspace is contained
in a \(G_\delta\) zero-dim subspace.
\end{lemma}

Recall from earlier
that zero-dim and zero-dim\(_X\) are equivalent in the context
of metrizable spaces. So it follows that
the equivalences in Corollary \ref{gdeltacor}
are guaranteed for metrizable spaces.

Likewise, and zero-dim and zero-ind are equivalent
for \(T_2\) strongly paracompact spaces, so for the following
theorem we relax our
notation to allow the use of ``zero dimensional'' (``zd'' for short)
and ``countable dimensional''.

\begin{theorem}\label{strongPara}
Let \(X\) be strongly paracompact and metrizable. Then the following
are equivalent.
    \begin{itemize}
        \item
            \(X\) is countable-dimensional.
        \item
            \(TWO\win G_c(\mathcal O,\mathcal O)\).
        \item
            \(TWO\markwin G_c(\mathcal O,\mathcal O)\).
        \item 
            \(TWO\win G_1(\mc O_{zd},\mc O)\).
        \item 
            \(TWO\markwin G_1(\mc O_{zd},\mc O)\).
        \item 
            \(ONE\win G_1(\mc N_{zd},\neg\mc O)\).
        \item 
            \(ONE\prewin G_1(\mc N_{zd},\neg\mc O)\).
    \end{itemize}
\end{theorem}

\begin{proof}
All equivalences except \(TWO\win G_c(\mathcal O,\mathcal O)\) are obtained from
Corollary \ref{gdeltacor} and the previous lemma.
This missing equivalence is obtained from
\cite{babinkostova2021countable}, whose proof assumes zero-ind is
equivalent to zero-dim, which is guaranteed by strong paracompactness.
\end{proof}

Due to use of the small-inductive characterization of zero dimension
in \cite{babinkostova2021countable},
we believe the following question remains open.

\begin{question}
Is strong paracompactness required in Theorem \ref{strongPara}?
\end{question}

\section{Inequivalence of
\texorpdfstring{\(G_c(\mc O,\mc O)\)}{G\_c(O,O)} and
\texorpdfstring{\(G_1(\mc O_{0\dim_X},\mc O)\)}{G\_1(O\_(0dim\_X),O)}
}

We begin with the following useful lemma, which generalizes
the classic result \cite[4.3]{Telgarsky1975}
of Telg{\'{a}}rsky on the equivalence of point-open and finite-open
games.

\begin{lemma}
Let \(\mc B\) be the closure of \(\mc A\) under finite unions.
Then the games \(G_1(\mc N_{\mc A},\neg\mc O)\) and
\(G_1(\mc N_{\mc B},\neg\mc O)\) are equivalent.
\end{lemma}
\begin{proof}
Note that since \(\mc A\subseteq\mc B\), \(\mc N_{\mc A}\subseteq\mc N_{\mc B}\).
Therefore we immediately have the following implications:
\begin{enumerate}
    \item \(ONE\prewin G_1(\mc N_{\mc A},\neg\mc O)\) implies
          \(ONE\prewin G_1(\mc N_{\mc B},\neg\mc O)\)
    \item \(ONE\win G_1(\mc N_{\mc A},\neg\mc O)\) implies
          \(ONE\win G_1(\mc N_{\mc B},\neg\mc O)\)
    \item \(TWO\win G_1(\mc N_{\mc B},\neg\mc O)\) implies
          \(TWO\win G_1(\mc N_{\mc A},\neg\mc O)\)
    \item \(TWO\markwin G_1(\mc N_{\mc B},\neg\mc O)\) implies
          \(TWO\markwin G_1(\mc N_{\mc A},\neg\mc O)\)
\end{enumerate}

For each \(B\in\mc B\), let \(N_B<\omega\) and \(A(B,n)\in\mc A\)
with \(B=\bigcup_{n\leq N_B}A(B,n)\). For convenience, let
\(A(B,n)=A(B,0)\) for \(n>N_B\), so
\(B=\bigcup_{n<\omega}A(B,n)\) (we will not be concerned with
the case \(B=\emptyset\) since ONE's moves are always improved by
choosing larger sets). We will assume strategies for and plays
by ONE choose an element of \(\mc A\) or \(\mc B\) directly each round, rather than
\(\mc N_{\mc A}\) or \(\mc N_{\mc B}\).

\textit{Converse 1.}
Suppose the sequence \(\tuple{B_0,B_1,\dots}\)
witnesses \(ONE\prewin G_1(\mc N_{\mc B},\neg\mc O)\).
Consider the predetermined strategy \[\tuple{A(B_0,0),\dots,A(B_0,N_{B_0}),
A(B_1,0),\dots,A(B_1,N_{B_1}),\dots}\] by ONE in
\(G_1(\mc N_{\mc A},\neg\mc O)\). Any response by TWO is of the form
\[\tuple{U(0,0),\dots,U(0,N_{B_0}),
U(1,0),\dots,U(1,N_{B_1}),\dots}\]
where \(A(B_i,j)\subseteq U(i,j)\).
Let \(U_i=\bigcup_{j\leq N_{B_i}}U(i,j)\); then
\(B_i\subseteq\bigcup_{j\leq N_{B_i}}U(i,j)=U_i\). Then
\(\tuple{U_0,U_1,\dots}\) is an unsuccessful response by TWO
against the winning predetermined strategy \(\tuple{B_0,B_1,\dots}\).
Thus \(\{U_i:i<\omega\}\) is not a cover, and it follows
that \(\{U(i,j):i<\omega,j\leq N_{B_i}\}\) is not a cover as well.
Thus the above predetermined strategy for ONE in \(G_1(\mc N_{\mc B},\neg\mc O)\)
is winning.

\textit{Converse 2.}
Suppose \(\tau\) is a strategy witnessing
\(ONE\win G_1(\mc N_{\mc B},\neg\mc O)\). Let
\(B_0=\tau(\tuple{})\) and \(m_0=N_{B_0}\). Then 
\(\tau(\tuple{})=\bigcup_{i\leq m_0}A(B_0,i)\).
Let \(\tau'(\tuple{V_{0,0},\dots,V_{0,i-1}})=A(B_0,i)\) for
\(i\leq m_0\). Note then that
\(\bigcup_{i\leq m_0}\tau'(\tuple{V_{0,k},\dots,V_{0,i-1}})
=B_0=\tau(\tuple{})\).

Suppose \(m_i<\omega\) is defined for \(i\leq p<\omega\)
and \(\tau'\)
has been defined for each initial segment of
\(\tuple{V_{0,0},\cdots,V_{0,m_0},\cdots,V_{p,0},\cdots,V_{p,m_p-1}}\)
where 
\[U_k=\bigcup_{i\leq m_k}V_{k,i}\supseteq\tau(\tuple{U_0,\cdots,U_{k-1}})\]
for \(k<p\) and
\[
\tau(\tuple{U_0,\cdots,U_{p-1}})=\bigcup_{i\leq m_p}\tau'(
\tuple{V_{0,0},\cdots,V_{0,m_0},\cdots,V_{p-1,0},\cdots,V_{p-1,i-1}})
.\]

Consider
\(\tuple{V_{0,0},\cdots,V_{0,m_0},\cdots,V_{p,0},\cdots,V_{p,m_p}}\),
and let 
\[
U_p=\bigcup_{i\leq m_p}V_{p,i}\supseteq
\tau(\tuple{U_0,\cdots,U_{p-1}})
.\]
Set
\(B_{p+1}=\tau(\tuple{U_0,\cdots,U_{p}})\) and \(m_{p+1}=N_{B_{p+1}}\).
Then 
\(\tau(\tuple{U_0,\cdots,U_p})=\bigcup_{i\leq m_{p+1}}A(B_{p+1},i)\).
Let \[\tau'(
\tuple{V_{0,0},\cdots,V_{0,m_0},\cdots,V_{p,0},\cdots,V_{p,m_p},
V_{p+1,0},\cdots,V_{p+1,i-1}}
)=A(B_{p+1},i)\] for
\(i\leq m_{p+1}\). Note then that
\[\bigcup_{i\leq m_{p+1}}\tau'(
\tuple{V_{0,0},\cdots,V_{0,m_0},\cdots,
V_{p+1,0},\cdots,V_{p+1,i-1}}
)=B_{p+1}=\tau(\tuple{U_0,\cdots,U_p}).\]
Finally, by following the construction we observe that any
counter-play against \(\tau'\) produces a corresponding
counter-play against \(\tau\) with the same union.
Thus since \(\tau\) is winning, so is \(\tau'\).

\textit{Converse 3.}
Suppose \(\tau\) is a strategy witnessing
\(TWO\win G_1(\mc N_{\mc A},\neg\mc O)\). Let
\[
\tau'(\tuple{B_0,\cdots,B_n}) =
\bigcup_{i\leq N_{B_n}}
\tau(\tuple{A(B_0,0),\cdots,A(B_0,N_{B_0}),\cdots,
A(B_n,0),\cdots,A(B_n,i)}).
\]
Then any counter-play \(\tuple{B_0,B_1,\cdots}\)
against \(\tau'\) corresponds to a counter-play
\[\tuple{A(B_0,0),\cdots,A(B_0,N_{B_0}),A(B_1,0),\cdots}\]
against \(\tau\) where both strategies cover
the same subset of the space. Therefore since
\(\tau\) is winning, so is \(\tau'\).

\textit{Converse 4.}
Suppose \(\tau\) is a strategy witnessing
\(TWO\markwin G_1(\mc N_{\mc A},\neg\mc O)\). Let
\(\theta:\omega^2\to\omega\) be a bijection. Then
we define the Markov strategy \(\tau'\) in
\(G_1(\mc N_{\mc B},\neg\mc O)\) by
\(\tau'(B,n)=\bigcup_{i< \omega}\tau(A(B,i),\theta(n,i))\).
Then if ONE chooses \(B_n\) during round \(n\) against \(\tau'\),
consider when ONE chooses \(A(B_n,i)\) during round \(\theta(n,i)\)
against \(\tau\). It follows that both plays result in TWO
constructing covers of the same subspace, so since \(\tau\) is winning,
so is \(\tau'\).
\end{proof}

By duality we have the following.

\begin{corollary}
Let \(\mc B\) be the closure of \(\mc A\) under finite unions.
Then the games \(G_1(\mc O_{\mc A},\mc O)\) and
\(G_1(\mc O_{\mc B},\mc O)\) are equivalent.
\end{corollary}

In \cite[Corollary 19]{BABINKOSTOVA20071971}, Babinkostova
cites an example of Pol \cite{pol1993spaces} where
\(ONE\notprewin G_c(\mc O,\mc O)\), but
\(ONE\prewin G_1(\mc O_{\mc A},\mc O)\), where \(\mc A\)
is the collection of ``finite-dimensional'' subsets of the
space. Since this example is separable and metrizable,
finite-dimensional here is equivalent to finite unions of
zero-dimensional (per your favorite definition) subsets 
\cite[4.1.17]{engelking1978dimension}. Therefore despite
the equivalence of \(TWO\markwin G_1(\mc O_{0\dim_X},\mc O)\)
and \(TWO\markwin G_c(\mc O,\mc O)\) among \(T_1\) spaces,
we have the following.

\begin{corollary}
The games
\(G_1(\mc O_{0\dim_X},\mc O)\) and
\(G_c(\mc O,\mc O)\) are not equivalent, even for
separable metrizable spaces.
\end{corollary}

\bibliographystyle{plain}
\bibliography{bibliography}

\end{document}